\documentclass[12pt,reqno]{amsart}
\usepackage{amsthm}
\usepackage{amssymb}
\usepackage{graphics}
\usepackage{tikz}
\usetikzlibrary{shapes,backgrounds,calc}
\usepackage{latexsym}
\usepackage{multicol}
\usepackage{verbatim,enumerate}
\usepackage{accents}
\usepackage{cite}
\usepackage{multirow}
\usepackage{bigstrut}
\usepackage{amsthm}
\usepackage{amssymb}
\usepackage{graphics}
\usepackage{tikz}
\usetikzlibrary{shapes,backgrounds,calc}
\usepackage{latexsym}
\usepackage{multicol}
\usepackage{verbatim,enumerate}
\usepackage{accents}
\usepackage{cite}
\usepackage{array}
\usepackage[colorlinks=true, linkcolor=blue, citecolor=blue, urlcolor=black]{hyperref}
\usepackage{hyperref}
\usepackage{amsmath, amscd,url}
\usepackage{multirow}
\usepackage{bigstrut}
\usepackage{longtable}

\usepackage{stackengine}

\usepackage{hyperref}
\usepackage{amsmath, amscd,url}
\usepackage{longtable}

\advance\textwidth by 1.3in \advance\oddsidemargin by -.6in \advance\evensidemargin by -.6in
\theoremstyle{definition}

\newtheorem{theorem}{Theorem}[section]
\newtheorem{lemma}[theorem]{Lemma}
\newtheorem{corollary}[theorem]{Corollary}
\theoremstyle{definition}
\newtheorem{definition}[theorem]{Definition}
\newtheorem{example}[theorem]{Example}

\theoremstyle{remark}
\newtheorem{remark}[theorem]{Remark}

\theoremstyle{definition}

\newcounter{cnt}
 \makeatletter
\def\mydggeometry{\makeatletter\dg@YGRID=1\dg@XGRID=20\unitlength=0.003pt\makeatother}
\makeatother \theoremstyle{remark}

\numberwithin{equation}{section}
\let\bwdg\bigwedge
\def\bigwedge{{\textstyle\bwdg}}

\newcommand{\Z}{\mathbb{Z}}

\newcommand{\nc}{\newcommand}
\newcommand{\rnc}{\renewcommand}

\nc{\cal}{\mathcal} \nc{\goth}{\mathfrak} \rnc{\bold}{\mathbf}

\nc\bomega{{\mbox{\boldmath $\omega$}}} \nc\bpsi{{\mbox{\boldmath $\Psi$}}}
\nc\balpha{{\mbox{\boldmath $\alpha$}}}
\nc\bpi{{\mbox{\boldmath $\pi$}}}
\nc\bvpi{{\mbox{\boldmath $\varpi$}}}
\nc\chara{\operatorname{ch}}

\nc\bxi{{\mbox{\boldmath $\xi$}}}
\nc\bmu{{\mbox{\boldmath $\mu$}}} \nc\bcN{{\mbox{\boldmath $\cal{N}$}}} \nc\bcm{{\mbox{\boldmath $\cal{M}$}}} \nc\blambda{{\mbox{\boldmath
			$\lambda$}}}\nc\bnu{{\mbox{\boldmath $\nu$}}}

\makeatletter
\def\section{\def\@secnumfont{\mdseries}\@startsection{section}{1}%
	\z@{.7\linespacing\@plus\linespacing}{.5\linespacing}%
	{\normalfont\scshape\centering}}
\def\subsection{\def\@secnumfont{\bfseries}\@startsection{subsection}{2}%
	{\parindent}{.5\linespacing\@plus.7\linespacing}{-.5em}%
	{\normalfont\bfseries}}
\makeatother

\nc{\Hom}{\operatorname{Hom}}
\nc{\mode}{\operatorname{mod}}
\nc{\End}{\operatorname{End}} \nc{\wh}[1]{\widehat{#1}} \nc{\Ext}{\operatorname{Ext}} \nc{\ch}{\text{ch}} \nc{\ev}{\operatorname{ev}}
\nc{\Ob}{\operatorname{Ob}} \nc{\soc}{\operatorname{soc}} \nc{\rad}{\operatorname{rad}} \nc{\head}{\operatorname{head}}

\nc{\Cal}{\cal} \nc{\Xp}[1]{X^+(#1)} \nc{\Xm}[1]{X^-(#1)}
\nc{\N}{{\bold N}}  \nc\boa{\bold a} \nc\bob{\bold b} \nc\boc{\bold c} \nc\bod{\bold d} \nc\boe{\bold e} \nc\bof{\bold f} \nc\bog{\bold g}
\nc\boh{\bold h} \nc\boi{\bold i} \nc\boj{\bold j} \nc\bok{\bold k} \nc\bol{\bold l} \nc\bom{\bold m} \nc\bon{\mathbb n} \nc\boo{\bold o}
\nc\bop{\bold p} \nc\boq{\bold q} \nc\bor{\bold r} \nc\bos{\bold s} \nc\boT{\bold t} \nc\boF{\bold F} \nc\bou{\bold u} \nc\bov{\bold v}
\nc\bow{\bold w} \nc\boz{\bold z}\nc\ba{\bold A} \nc\bb{\bold B} \nc\bc{\mathbb C} \nc\bd{\bold D} \nc\be{\bold E} \nc\bg{\bold
	G} \nc\bh{\bold H} \nc\bi{\bold I} \nc\bj{\bold J} \nc\bk{\bold K} \nc\bl{\bold L} \nc\bm{\bold M} \nc\bn{\mathbb N} \nc\bo{\bold O} \nc\bp{\bold
	P} \nc\bq{\bold Q} \nc\br{\bold R} \nc\bs{\bold S} \nc\bt{\bold T} \nc\bu{\bold U} \nc\bv{\bold V} \nc\bw{\bold W} \nc\bz{\mathbb Z} \nc\bx{\bold
	x} \nc\KR{\bold{KR}} \nc\rk{\bold{rk}} \nc\het{\text{ht }}

\nc\toa{\tilde a} \nc\tob{\tilde b} \nc\toc{\tilde c} \nc\tod{\tilde d} \nc\toe{\tilde e} \nc\tof{\tilde f} \nc\tog{\tilde g} \nc\toh{\tilde h}
\nc\toi{\tilde i} \nc\toj{\tilde j} \nc\tok{\tilde k} \nc\tol{\tilde l} \nc\tom{\tilde m} \nc\ton{\tilde n} \nc\too{\tilde o} \nc\toq{\tilde q}
\nc\tor{\tilde r} \nc\tos{\tilde s} \nc\toT{\tilde t} \nc\tou{\tilde u} \nc\tov{\tilde v} \nc\tow{\tilde w} \nc\toz{\tilde z} \nc\woi{w_{\omega_i}}

\begin{document}

\setcounter{section}{0}
\setcounter{tocdepth}{1}

%%%%%%%%%%%%%%%%%%%%%%%%%%%%%%%%%%%%%%%%%%%%

\title{irreducible $\Z_+$-modules over $\Z[\sqrt[m]{d_1},\sqrt[m]{d_2},\ldots, \sqrt[m]{d_r}]$}
%\dedicatory{Dedicated to Professor Sudesh Kaur Khanduja on her 76$^{th}$ birthday}
\author[Surender Kumar]{Surender Kumar}
\author[Sumandeep Kaur]{Sumandeep Kaur}
%\author{Qing-Wen Wang}\thanks{}
\address[Surender Kumar] {Mathematics Division, School of advanced scinces and Languages,  VIT Bhopal University, Sehore, 466114, INdia}
\address[Sumandeep Kaur\footnote{Corresponding author, Email: suman@shu.edu.cn}]{Department of Mathematics, Shanghai University, China}
%\address[Qing-Wen Wang\footnote{Corresponding author}]{Department of Mathematics, Shanghai University, China}
\email[Surender Kumar]{surenderkumar@vitbhopal.ac.in}
\email[Sumandeep Kaur]{suman@shu.edu.cn}
%\email[Qing-Wen Wang]{wqw@shu.edu.cn}


\begin{abstract}
 Let $d_1,d_2,\ldots, d_r\ge 2$ be  $m$-th power free positive integers. In this paper, we provide explicit representations of all irreducible $\mathbb{Z}_+$-modules over $\mathbb{Z}[\sqrt[m]{d_1},\sqrt[m]{d_2},\ldots,\sqrt[m]{d_r}]$, where $\mathbb{Z}_+$ denotes the semiring of non-negative integers.

\end{abstract}
\maketitle
\section{Introduction}
Tensor categories serve as categorical analogues of groups and rings, and they arise naturally in non-commutative algebra and representation theory. A key invariant associated with a tensor category is its $\mathbb{Z}_+$-ring, which provides a powerful tool for analyzing $\mathbb{Z}_+$-modules and their representations. This notion was originally introduced by Lusztig~\cite{Lusztig1987}. For a detailed exposition of the definitions and fundamental properties of $\mathbb{Z}_+$-rings, we refer the reader to the works of Etingof and Khovanov~\cite{ETIN} and Ostrik~\cite{OST1}. The theory of fusion categories offers a natural framework for studying based rings and their representation theory (see~\cite{NO}). Moreover, Etingof and Khovanov demonstrated that $\mathbb{Z}_+$-representations arise from tensor categories and module categories, and that their indecomposable forms can be classified via Dynkin diagrams through realizations by non-negative integer matrices~\cite{ETIN}.

In~\cite{OST1}, Ostrik developed a systematic theory of module categories over rigid monoidal categories and demonstrated that their Grothendieck groups naturally inherit the structure of based $\mathbb{Z}_+$-modules, thereby providing a conceptual foundation for non-negative integer matrix representations and their classification.

Recall that if a $\mathbb{Z}_+$-module $M$ admits no nontrivial $\mathbb{Z}_+$-submodules, then $M$ is called irreducible~\cite{OST1}. For a finite group $G$, the irreducible $\mathbb{Z}_+$-modules over the group ring $\mathbb{Z}[G]$ were characterized in~\cite{ETIN}. Ostrik proved that for a $\mathbb{Z}_+$-ring of finite rank, there exist only finitely many inequivalent irreducible $\mathbb{Z}_+$-modules~\cite{OST1}. In this direction, determining the structure of inequivalent irreducible $\mathbb{Z}_+$-modules over a given ring of finite rank is an important problem. In general, the classification of irreducible $\mathbb{Z}_+$-modules over an arbitrary $\mathbb{Z}_+$-ring is a difficult open problem, and explicit classifications are known only in certain cases (see~\cite{CHI1,CHI2,CHI}).

Let $d_{1}, d_{2}, \ldots, d_{r}$ be pairwise relatively prime, positive  and $m$-th power free integers with
$d_{j} \ge 2$ for all $1 \le j \le r$.
In this article, we study the representation of irreducible $\mathbb{Z}_{+}$-modules over the domain $R = \mathbb{Z}[\sqrt[m]{d_{1}}, \sqrt[m]{d_{2}}, \ldots, \sqrt[m]{d_{r}}].$ As a main result, we provide a complete classification of irreducible $\mathbb{Z}_{+}$-modules over $R$. Our approach combines elementary matrix-theoretic techniques with combinatorial properties of nonnegative integer matrices. Also, we show that every irreducible $\mathbb{Z}_{+}$-module of $R$ has rank $m^{r}$.
%Furthermore, all irreducible $\mathbb{Z}_{+}$-modules over $R$ are explicitly constructed. 

\section{Preliminaries}	
Let $\mathbb{Z}_{+}=\mathbb N \cup \{0\}$, and $M_n(\mathbb{Z}_{+})$ be the semiring consisting of  $n\times n$ matrices  with entries from $\mathbb Z_+$. In this section, first we recall the definitions of $\mathbb{Z}_{+}$-ring and $\mathbb{Z}_{+}$-module \cite{Lusztig1987}.

\begin{definition}
	Let $R$ be a unital domain such that $R$ is a free  $\mathbb{Z}$-module with basis
	$\{x_1=1,x_2,\ldots,x_n\}$. If for  $1 \le i,j \le n,$ we can write
	
	$$x_i x_j = \sum_{l=1}^{n} u_{ij}^l x_l$$
	 for some $u_{ij}^l \in \mathbb{Z}_{+}$, 
	then $R$ is called a unital $\mathbb{Z}_{+}$-ring. Here $N_R = (u_{ij}^l)$ is the structure constants of $R$, and $n$ is the rank of $R$.
\end{definition}

\begin{definition}\label{D2}
	Let $R$ be a unital $\mathbb{Z}_{+}$-ring with basis $\{x_1,\ldots,x_n\}$. An $R$-module $M$ is called a $\mathbb{Z}_{+}$-module of $R$ with $\mathbb{Z}_{+}$-basis
	$\{\omega_1,\omega_2,\ldots,\omega_t\}$ if
	for $1 \le i \le n,\; 1 \le j \le t,$\[
	x_i \omega_j = \sum_{l=1}^{t} a_{ij}^l \omega_l,
	\quad a_{ij}^l \in \mathbb{Z}_{+}.
	\]
	Here $t$ is called the rank of $M$.
	
	If an additive subgroup $M'$ of $M$ spanned by a subset  $\{\omega_{j_1},\ldots,\omega_{j_s}\} \subseteq \{\omega_1,\omega_2,\ldots,\omega_t\}$ is an $R$-submodule of $M$, then we say that $M'$ is a $\mathbb{Z}_{+}$-submodule of $M$. $M'$ is also called the $\mathbb{Z}_{+}$-submodule generated by $\{\omega_{j_1},\ldots,\omega_{j_s}\}$, and it will be denoted by $(\omega_{j_1},\ldots,\omega_{j_s})$. A $\mathbb{Z}_{+}$-module $M$ is called irreducible if the $\mathbb{Z}_{+}$-submodules of $M$ are only ${0}$ and $M$.
\end{definition}

Note that $M$ is a $\mathbb{Z}_{+}$-module of a ring $R$ if and only if there exists a semiring homomorphism  from $R$ to $M_n(\mathbb{Z}_{+})$, where $n$ is the rank of $M$.

\begin{definition}
	Let $R$ be a $\mathbb Z_+$ ring with basis $\{x_1,x_2,\cdots,x_n\}$. Let $M$ and $ M'$ be $\mathbb{Z}_{+}$-modules of $R$. Assume that
	$\{\omega_1,\ldots,\omega_t\}$ and $\{\omega'_1,\ldots,\omega'_t\}$ are $\mathbb{Z}_{+}$-bases of $M$ and $M'$, respectively.  Next, assume
	\[
	x_i \omega_j = \sum_{k=1}^{t} a_{ij}^k \omega_k
	\quad \text{and} \quad
	x_i \omega'_j = \sum_{k=1}^{t} b_{ij}^k \omega'_k.
	\]
	If there exists an $R$-module homomorphism $\phi : M \to M'$ such that
	$\phi(\omega_j) = \omega'_{\pi(j)}$, where  $\pi$ is permutation on $\{1,2,\cdots,t\}$, then $M$ is said to be isomorphic to $M'$, and denoted by $M \cong M'$.
\end{definition}

\begin{remark}\label{re1}
	Assume $\phi : M \to M'$ is an isomorphism of $\mathbb{Z}_{+}$-modules of $R$. Then
	\[
	\phi(x_i \omega_j)
	= \sum_{k=1}^{t} a_{ij}^k \omega'_{\pi(k)}
	= \sum_{l=1}^{t} b_{i\pi(j)}^{\,l} \omega'_l
	= x_i \omega'_{\pi(j)}.
	\]
	
	Let $A_i := (a_{ij}^k)_{t \times t}$ and $B_i := (b_{ij}^k)_{t \times t}$,
	$1 \le i \le n$, $1 \le j,k \le t$. Hence, two $\mathbb{Z}_{+}$-modules $M$ and $M'$
	are isomorphic if and only if there exists a permutation $\pi$ on $\{1,2,\cdots,t\}$ such that
	\[
	a_{ij}^k = b_{i\pi(j)}^{\,\pi(k)}.
	\]
	Equivalently, there exists a permutation matrix $P$ such that
	\[
	B_i = P A_i P^{-1} \quad \text{for all } 1 \le i \le n.
	\]
In this article, we use  the matrix method to classify the isomorphism
	classes of the irreducible $\mathbb{Z}_{+}$-modules over the domain $\mathbb{Z}[\sqrt[m]{d_1},\ldots,\sqrt[m]{d_r}]$, where $d_1,\ldots,d_r$ are pairwise
relatively prime, positive and $m$-th power free integers with $d_j \ge 2$ for all
$1 \le j \le r$.
\end{remark}
The following lemma will be used in the proof of our main theorem.
\begin{lemma}\label{lemma su}
	Let $d_1,d_2$ be two coprime positive $m$-th power free integers.  Let $$B_1 =
		\begin{pmatrix}
	0   & b_{11} &0 &0&\cdots&0\\
	0&0&b_{12} &0&\cdots &0\\
    \vdots & \vdots &\vdots&\vdots  & \vdots & \vdots\\
    0&0&0&0&\cdots &b_{1{(m-1)}}\\
	b_{1m}&0&0&0 &\cdots & 0
\end{pmatrix},$$ and  $$B_2 =
		\begin{pmatrix}
	0   & b_{21} &0 &0&\cdots&0\\
	0&0&b_{22} &0&\cdots &0\\
    \vdots & \vdots &\vdots&\vdots  & \vdots & \vdots\\
    0&0&0&0&\cdots &b_{2{(m-1)}}\\
	b_{2m}&0&0&0 &\cdots & 0
\end{pmatrix},$$ with
	 $b_{i1},b_{i2}, \ldots, b_{im}$ are positive integers such that
$b_{i1}b_{i2}\cdots b_{im} = d_1$ and not all $b_{ij}$ are equal for $i=1,2$, and $j=1,2,\ldots m$.  Assume $A=[x_{ij}]\in M_m(\mathbb Z_+)$ satisfies
	\[
	B_1A = AB_2,
	\]
	and all non-zero entries of $A$ are divisors of $d_2$. Then the following hold:
	\begin{itemize}
		\item[(1)] If $B_1\neq B_2$, then $A=0$.
		\item[(2)] If $B_1=B_2$, then $A$ is a scalar matrix.
	\end{itemize}
\end{lemma}
\begin{proof}
	Write $A=
	\begin{pmatrix}
		x_{11} & x_{12} & \cdots & x_{1m}\\
		x_{21} & x_{22} & \cdots & x_{2m}\\
		\vdots & \vdots & \ddots & \vdots\\
		x_{m1} & x_{m2} & \cdots & x_{mm}
	\end{pmatrix}.$ The condition $B_1A = AB_2$ gives equality of the corresponding matrix entries, so we get:
\[
\begin{aligned}
	b_{11}x_{21}= b_{2m}x_{1m}, &~ ~ b_{11}x_{22} = b_{21}x_{11}, & \cdots &
	b_{11}x_{2m} = b_{2(m-1)}x_{1(m-1)}\\
	b_{12}x_{31}= b_{2m}x_{2m}, &~ ~ b_{12}x_{32} = b_{21}x_{21}, & \cdots &
	b_{12}x_{3m} = b_{2(m-1)}x_{2(m-1)}\\
	\vdots & \vdots & \ddots & \vdots\\
	b_{1(m-1)}x_{m1}= b_{2m}x_{(m-1)m}, &~ ~
	b_{1(m-1)}x_{m2} = b_{21}x_{(m-1)1}, & \cdots &
	b_{1(m-1)}x_{mm} = b_{2(m-1)}x_{(m-1)(m-1)}\\
	b_{1m}x_{11}= b_{2m}x_{mm}, &~ ~
	b_{1m}x_{12} = b_{21}x_{m1}, & \cdots &
	b_{1m}x_{1m} = b_{2(m-1)}x_{m(m-1)}.
\end{aligned}
\tag{$\ast$}
\]
	
	\noindent Since all entries of $A$ are non-negative integers and all non-zero entries divide $d_2$, while each $b_{ij}$ divides $d_1$, the condition
	$\gcd(d_1,d_2)=1$
	implies that if $x_{ij}\neq0$, then the coefficients multiplying it on both sides of each equation in $(\ast)$ must be equal.

	\noindent (1) Assume $B_1\neq B_2$. Then at least one of
	$b_{11}\neq b_{21},\quad b_{12}\neq b_{22},\quad \ldots,\quad b_{1m}\neq b_{2m}$
	holds. From the system $(\ast)$, any non-zero entry $x_{ij}$ would force equality of the
	corresponding coefficients, which is impossible. Hence all $x_{ij}=0$, and thus
	$A=0.$

	\medskip

	\noindent (2) Assume $B_1=B_2$. Then
	$b_{11}=b_{21},\quad b_{12}=b_{22},\quad \ldots,\quad b_{1m}=b_{2m}.$ Clearly $x_{11}=x_{22}=\cdots=x_{mm}.$ Since not all $b_{11},b_{12},\ldots,b_{1m}$ are equal, the system $(\ast)$ forces all off-diagonal entries of $A$ to be zero. Therefore, $A$ is a scalar matrix.
	\end{proof}
\section{The main theorem}
In this section, we present a complete classification of irreducible \(\mathbb{Z}_{+}\)modules over the domain
$\mathbb{Z}\big[\sqrt[m]{d_1}, \ldots, \sqrt[m]{d_r}\big],$
where $d_1, \ldots, d_r$ are pairwise relatively prime, positive and $m$-th power free integers integers and greater than $1$.
For an identity matrix $I_n$ of order $n,$ let \(E(i,j)\) denote the elementary matrix obtained by interchanging the \(i\)-th and \(j\)-th rows of \(I_n\).  

\begin{theorem}\label{susu}
	Let $d$ be a positive, and $m$-th power free integers integer. Let $A = [a_{ij}]_{n\times n}$ be a non-negative integer
	matrix such that
	$A^m = d I_n$ and $n>m$.
	Then $n = mt$, and there exists a permutation matrix
	$P$ %= E(i_1,j_1) E(i_2,j_2) \cdots E(i_l,j_l)$
	such that
	$PAP^{-1} =\mathrm{diag}(A_1, A_2, \ldots, A_n)$, where for $1\le k\le t,$
	%\begin{pmatrix}A_1 &        &        \\& A_2    &        \\&        & \ddots \\&        &        & A_t\end{pmatrix}, ~  ~ 
	$A_k =
		\begin{pmatrix}
	0   & b_{k1} &0 &0&\cdots&0\\
	0&0&b_{k2} &0&\cdots &0\\
    \vdots & \vdots &\vdots&\vdots  & \vdots & \vdots\\
    0&0&0&0&\cdots &b_{k{(m-1)}}\\
	b_{km}&0&0&0 &\cdots & 0
\end{pmatrix},$ with
	 $b_{k1},b_{k2}, \ldots, b_{km}$ are positive integers such that
	$b_{k1}b_{k2}\cdots b_{km} = d.$
	%Moreover, there is exactly one non-zero element in each row and each column of $A$.
\end{theorem}
\begin{proof}
	Let
	$A=[a_{ij}]_{n\times n}$ be a matrix with rows $\alpha_1, \alpha_2,\ldots,\alpha_n$ and columns $\beta_1,\beta_2, \ldots \beta_n$.
	Assume $\gamma^{k}_1,\gamma^{k}_2,\ldots,\gamma^{k}_n$ are rows of $A^k$ for all $2\le k\le m.$ Given that  $A^m = dI_n$. So we have $(\gamma^{m-1}_1,\gamma^{m-1}_2,\ldots,\gamma^{m-1}_n)^T(\beta_1,\beta_2,\ldots,\beta_n)=dI_n,$ that is,
	$$A^m=A^{m-1}A=\begin{pmatrix}
		(\gamma^{m-1}_1,\beta_1) & (\gamma^{m-1}_1,\beta_2) & \cdots & (\gamma^{m-1}_1,\beta_n) \\
		(\gamma^{m-1}_2,\beta_1) & (\gamma^{m-1}_2,\beta_2) & \cdots & (\gamma^{m-1}_2,\beta_n) \\
		\vdots             & \vdots             &        & \vdots             \\
		(\gamma^{m-1}_n,\beta_1) & (\gamma^{m-1}_n,\beta_2) & \cdots & (\gamma^{m-1}_n,\beta_n)
	\end{pmatrix}
	= d I_n,$$ where $(\gamma^{m-1}_i,\beta_j) = \sum_{k=1}^n (\gamma^{m-2}_i,\beta_k) a_{kj}$ for all $1\le i,j\le n.$
	
	\noindent Since $(\gamma^{m-1}_1,\beta_1) = \sum_{k=1}^n (\gamma^{m-2}_1,\beta_k) a_{k1}=d \ne 0$, there exists an integer
	$1 \le i_1 \le n$ such that $(\gamma^{m-2}_1,\beta_{i_1})\ne0$ and $a_{i_1 1} \ne 0$.
	But $	(\gamma^{m-1}_1,\beta_j)=\sum_{k=1}^n (\gamma^{m-2}_1,\beta_k)a_kj=0,$ for $2\le j\le n$ implies that $(\gamma^{m-2}_1,\beta_{i_1})a_{i_1j}=0$. Therefore for all $2\le j\le n,$ we have $a_{i_1j}=0.$ Thus $\alpha_{i_1}=(a_{i_11},0,\ldots,0).$ If $i_1=1,$ then $A^m=dI_n$ implies that $a^m_{11}=d$ which is a contradiction. Therefore $i_1>1$. Also $A^m=AA^{m-1}$, shows that for any $j\ne i_1$ we have $(\gamma_{i_1}^{m-1},\beta_j)=\displaystyle\sum_{k=1}^n a_{i_1k}(\gamma^{m-2}_k,\beta_{j})=0.$ Therefore $(\gamma^{m-2}_1,\beta_{j})=0$ for all $j\ne i_1.$ 

\noindent Now $(\gamma^{m-2}_{1},\beta_{i_1})=\displaystyle\sum _{k=1}^n(\gamma^{m-3}_{1},\beta_{k})a_{ki_1}\ne 0$ implies that  $(\gamma^{m-3}_{1},\beta_{i_2})a_{i_2i_1}\ne 0$ for some $1\le i_2\le n,$ that is, $(\gamma^{m-3}_{1},\beta_{i_2})\ne 0 $ and $a_{i_2i_1}\ne 0.$ For any $1 \le j \le n$ with $j \ne i_1$ the condition, $\displaystyle\sum _{k=1}^n(\gamma^{m-3}_{1},\beta_k)a_{kj}=0$ gives $a_{i_2j}=0$ for all $j\ne i_1$. Therefore $\alpha_{i_2}=(0,\ldots,a_{i_2i_1},\ldots,0).$ If $ i_2=i_1,$ then  $ i_2=i_1=1.$ Therefore $a^m_{11}=d$ which is a contradiction. So $i_1>1$ and $i_2\ne i_1.$ Now if $A^2=[b_{ij}]$ then $A^m=A^2A^{m-2}$ shows that for any $j\ne i_2$ we have $(\gamma_{i_2}^{m-1},\beta_j)=\displaystyle\sum_{k=1}^n b_{i_2k}(\gamma^{m-3}_k,\beta_{j})=0.$   Since $a_{i_2i_1}$ and $a_{i_11}$ are non-zero, we see that $(\gamma^{m-3}_1,\beta_{j})=0$ for all $j\ne i_2.$ 

In a similar way, we obtain $1\le i_3, i_4,\ldots,i_{m-1}\le n$ such that $i_j\ne i_k$ for $1\le j,k\le n$ and $\alpha_{i_k}=(0,\ldots,0,a_{i_ki_{k-1}},0\ldots,0) ~ ~ ~ \text{ for all } 3\le k\le m-1$ with $a_{i_ki_{k-1}}\ne0.$ Next, if $1\le j\le n$ then  $\displaystyle\sum_{k_1=1}^n\displaystyle\sum_{k_2=1}^n\cdots\displaystyle\sum_{k_{m-1}=1}^na_{i_{m-1}k_1}a_{k_1k_2}\cdots a_{k_{m-1}j}=\begin{cases}
d, & j=i_{m-1} \\
0, & j\ne i_{m-1}
\end{cases}$. Therefore the condition $a_{i_{m-1}i_{m-2}}a_{i_{m-2}i_{m-3}}\cdots a_{i_11}a_{1j}=0$ implies that $a_{1i_{m-1}}\ne0$ and $a_{1j}=0$ for all $j\ne i_{m-1}.$ Thus $\alpha_1=(0,\ldots,0,a_{1i_{m-1}},0,\ldots,0).$ If we set $b_1=a_{1i_{m-1}}, b_j=a_{i_ji_{j-1}}$ for $2\le j\le m-1$ and $b_m=a_{i_1 1}$, it is obvious that $\displaystyle\prod_{j=1}^nb_j=d.$ Now there exist some elementary matrices $E_1,\ldots,E_r$ such that we obtain a permutation matrix $P=E_1E_2 \cdots E_r$ and without loss of generality we can set a new matrix $B=PAP^{-1}$ as $$B=
	\begin{pmatrix}
		0      & b_1 &0     & \cdots & 0 & 0 & 0 &\cdots & 0\\
		0 & 0      & b_2      & \cdots & 0 & 0 & 0 &\cdots & 0 \\
        \vdots & \vdots & \vdots & \cdots &\vdots& \vdots & \vdots &\cdots & \vdots\\
        0&0&0&\cdots & b_{m-1}& 0 & 0 &\cdots & 0\\
		b_m &0& 0  & \cdots & 0 & 0 & 0 &\cdots & 0\\
        &&&&A_1&&&&
		%a_{{m+1}1} & a_{{m+1}2} & a_{{m+3}3} & \cdots & a_{4m} & a_{{m+1}{m+1}} & a_{{m+1}{m+2}} &\cdots & a_{{m+1}n}\\\vdots & \vdots   & \vdots &  \cdots & \vdots & \vdots & \vdots &\cdots & \vdots \\a_{n1} & a_{n j_0} & a_{n3} & \cdots & a_{nm} & a_{n{m+1}}&a_{n{m+2}}&\cdots& a_{nn}
	\end{pmatrix},$$ where $A_1$ represent the $(n-m)$ rows of $A$ which are different from $\alpha_1,\alpha_{i_1},\ldots\alpha_{i_{m-1}}.$ Hence 
	$B^m = PA^mP^{-1} = d I_n.$ Similarly, there exist
	$m+1 \le j_1<j_2<\ldots<j_{m} \le n$ such that $\displaystyle\prod_{i=1}^mb_{j_i} = d.$
	By induction, for any $1 \le k \le n$, we know that exactly one coefficient of
	the vectors $\alpha_k$ and $\beta_k$ is non-zero, and that there exists a unique
	$1 \le e_1<e_2<\ldots e_{m-1} \le n$ $(k \ne e_i)$ satisfying $\displaystyle\prod_{i=1}^ma_{ke_1}a_{e_1e_2}\cdots a_{e_{m-1}k} = d$.
	This shows that $n$ is multiple of $m$. Thus, there exists a permutation matrix
	$P = E(i_1,i_2) E(i_2,i_3) \cdots E(i_s,i_t)$
	such that $ PAP^{-1}=
\begin{pmatrix}
	A_1 &        &        \\
	& A_2    &        \\
	&        & \ddots \\
	&        &        & A_t
\end{pmatrix}
\text{with }
A_k =
\begin{pmatrix}
	0   & b_{k1} &0 &0&\cdots&0\\
	0&0&b_{k2} &0&\cdots &0\\
    \vdots & \vdots &\vdots&\vdots  & \vdots & \vdots\\
    0&0&0&0&\cdots &b_{k{(m-1)}}\\
	b_{km}&0&0&0 &\cdots & 0
\end{pmatrix},$
where $b_{k1}, b_{k2},\ldots,b_{km}$ are positive integers such that their product is $d$ for all
$1 \le k \le t$.
\end{proof}

\begin{theorem}\label{suman}
Let $R = \mathbb{Z}[\sqrt[m]{d}]$, where $d$ is a positive and $m$-th power free integer. If $M$ is an irreducible $\mathbb{Z}_{+}$-module of $R$, then the following hold:
\begin{enumerate}
\item[$(i)$] $\mathrm{Rank}(M) = m,$
\item[$(ii)$] There exist $b_1,b_2,\ldots,b_m\in \Z_+$ such that $\displaystyle\prod_{i=1}^mb_i=d$ and  $$M ~  \cong ~  \mathbb{Z}\langle \sqrt[m]{b_2b_3^2\cdots b^{m}_{m-1}}, \sqrt[m]{b_3b_4^2\cdots b^{m-2}_{m}b_1^{m-1}}, \ldots ,\sqrt[m]{b_1b_2^2\cdots b_{m-1}^{m-1}} \rangle$$
as  $\mathbb{Z}_{+}$-modules of $R$. 
%\item The number of irreducible $\mathbb Z_+$-modules upto isomorphism  of $R$ is $$\frac{1}{m!}\displaystyle\sum_{\alpha\in T} \binom{m}{\alpha'}\right\displaystyle\prod_{i=1}^{m-1}X_i^{r_i},$$ where $X_i,\alpha,\alpha'$ and $T$ are same as given in Lemma \ref{ns}.
\end{enumerate}
\end{theorem}

\begin{proof}
	Let $\{m_1,\ldots,m_n\}$ be a $\Z_+$ basis of $M$. Let
	$A = [a_{ij}]_{n\times n}$ be the non-negative integer matrix determined by $\sqrt[m]{d}$. 
	Then it follows from Theorem \ref{susu} that there exists a permutation matrix $P$
	such that $PAP^{-1} =\mathrm{diag}(A_1, A_2, \ldots, A_n)$, where for $1\le k\le t,$
	%\begin{pmatrix}A_1 &        &        \\& A_2    &        \\&        & \ddots \\&        &        & A_t\end{pmatrix}, ~  ~ 
	$A_k =
		\begin{pmatrix}
	0   & b_{k1} &0 &0&\cdots&0\\
	0&0&b_{k2} &0&\cdots &0\\
    \vdots & \vdots &\vdots&\vdots  & \vdots & \vdots\\
    0&0&0&0&\cdots &b_{k{(m-1)}}\\
	b_{km}&0&0&0 &\cdots & 0
\end{pmatrix},$ with
	 $b_{k1},b_{k2}, \ldots, b_{km}$ are positive integers such that
	$b_{k1}b_{k2}\cdots b_{km} = d.$  Then $M$ can be decomposed
as a direct sum of $\Z_+$-submodules of rank $m$ by Remark \ref{re1}. However, $M$ is irreducible, so $\mathrm{Rank}(M)=m$.  Therefore       $
	PAP^{-1}  =
\begin{pmatrix}
	0   & b_{1} &0 &0&\cdots&0\\
	0&0&b_{2} &0&\cdots &0\\
    \vdots & \vdots &\vdots&\vdots  & \vdots & \vdots\\
    0&0&0&0&\cdots &b_{m-1}\\
	b_{m}&0&0&0 &\cdots & 0
\end{pmatrix}, $
	where $b_1,\ldots,b_{m-1}$ and $b_m$ are positive integers such that  $\displaystyle\prod_{i=1}^nb_i = d$. Let $\mu_1=\sqrt[m]{b_2b_3^2\cdots b^{m-1}_{m}}, ~ \mu_2=\sqrt[m]{b_3b_4^2\cdots b^{m-2}_{m}b_1^{m-1}}$  and   we let $$\mu_i=\sqrt[m]{b_{i+1}b_{i+2}^2\cdots b^{m-i}_{m}b^{m-i+1}_1b_2^{m-i+2}\cdots b^{m-1}_{i-1}}, ~ ~ ~ ~ ~ ~ \text{ for } 3\le i\le m.$$ One can easily note that $$\sqrt[m]{d}\mu_1=b_{m}\mu_m, \sqrt[m]{d}\mu_2=b_{1}\mu_1 \text{ and } \sqrt[m]{d}\mu_i=b_{i-1}\mu_i \text{ for all } 3\le i\le m.$$
	Hence by Definition \ref{D2}, $M ~  \cong ~  \mathbb{Z}\langle \sqrt[m]{b_2b_3^2\cdots b^{m-1}_{m}}, \sqrt[m]{b_3b_4^2\cdots b^{m-2}_{m-1}b_1^{m-1}}, \ldots ,\sqrt[m]{b_1b_2^2\cdots b_{m-1}^{m-1}} \rangle.$ 
    %Thus, by corollary \ref{nr}, the ring $\mathbb{Z}[\sqrt[m]{d}]$ has exactly $\frac{1}{m!}\displaystyle\sum_{\alpha\in T} \binom{m}{\alpha'}\right\displaystyle\prod_{i=1}^{m-1}X_i^{r_i}$ irreducible $\mathbb{Z}_{+}$-modules up to isomorphism.
\end{proof}
The following corollary deals with the counting of the irreducible $\Z_+$ modules over the ring $\Z[\sqrt[m]{d}],$ where  $d$ is a square-free positive integer. 
\begin{corollary}\label{aux}
Let $m\ge 2$ be a positive integer and $d=\displaystyle\prod_{i=1}^{n}p_i$ be a square-free positive integer. If $M$ is an irreducible $\mathbb{Z}_{+}$-module of $\mathbb{Z}[\sqrt[m]{d}]$, then the following hold:
\begin{enumerate}
\item[$(i)$] $\mathrm{Rank}(M) = m,$
\item[$(ii)$] There exist $b_1,b_2,\ldots,b_m\in \Z_+$ such that $\displaystyle\prod_{i=1}^mb_i=d$ and  $$M ~  \cong ~  \mathbb{Z}\langle \sqrt[m]{b_2b_3^2\cdots b^{m-1}_{m}}, \sqrt[m]{b_3b_4^2\cdots b^{m-2}_{m-1}b_1^{m-1}}, \ldots ,\sqrt[m]{b_1b_2^2\cdots b_{m-1}^{m-1}} \rangle$$
as  $\mathbb{Z}_{+}$-modules of $ \mathbb{Z}[\sqrt[m]{d}]$. 
\end{enumerate}
\end{corollary}

\begin{theorem}\label{su}
Let $m$ be a positive integer. Assume that $d_{1}, d_{2}, \ldots, d_{r}$ are pairwise relatively prime, positive and $m$-th power free integers
with $d_{j} \ge 2$ for all $1 \le j \le r$. Let
$R = \mathbb{Z}[\sqrt[m]{d_{1}},\sqrt[m]{d_{2}}, \ldots, \sqrt[m]{d_{r}}],$
and let $M$ be an irreducible $\mathbb{Z}_{+}$-module of $R$. Then the following hold. \begin{enumerate}
\item[$(i)$] $\mathrm{Rank}(M) = m^{r}$.
\item[$(ii)$] For any $1 \le j \le r$, $M$ is isomorphic, as a $\mathbb{Z}_{+}$-module of
$\mathbb{Z}[\sqrt[m]{d_{j}}]$, to a direct sum
$M \cong \bigoplus_{i=1}^{m^{r-1}} M_{i}.$ Moreover, there exist positive integers $b_{i1}, b_{i2},\ldots, b_{im}$ such that
$b_{i1} b_{i2}\cdots b_{im} = d_{i}$ and for all $1 \le i \le m^{r-1}$, we have
$$M_{i} \cong \mathbb{Z}\langle \sqrt[m]{b_{i2}b_{i3}^2\cdots b_{im}^{m-1}}, \sqrt[m]{b_{i3}b_{i4}^2\cdots b^{m-2}_{im}b^{m-1}_{i1}}, \ldots ,\sqrt[m]{b_{i1}b_{i2}^2\cdots b_{i(m-1)}^{m-1}} \rangle.$$ 
\end{enumerate}\end{theorem}
\begin{proof}
 For any $1 \le j \le r$, choose a
$\mathbb{Z}_{+}$-basis of $M$ such that the non-negative integer matrix
determined by $\sqrt[m]{d_j}$ is of the form
$B=
\begin{pmatrix}
B_1 &        &        \\
    & B_2    &        \\
    &        & \ddots \\
    &        &        & B_t
\end{pmatrix},$ where
$B_k=
\begin{pmatrix}
	0   & b_{k1} &0 &0&\cdots&0\\
	0&0&b_{k2} &0&\cdots &0\\
    \vdots & \vdots &\vdots&\vdots  & \vdots & \vdots\\
    0&0&0&0&\cdots &b_{k(m-1)}\\
	b_{km}&0&0&0 &\cdots & 0
\end{pmatrix} $ and $b_{k1}b_{k2}\cdots b_{km}=d_j$ for all $1 \le k \le t$.
Moreover, without loss of generality, we may further assume that each $B_i$ matrix occurring $e_i$ times with $e_1+e_2+\cdots+e_s=t$. Let $A$ be the non-negative integer matrix determined by
$\sqrt[m]{d_e}$ for some $1 \le e \neq j \le r$.
Since the ring $R$ is commutative, we have
$\sqrt[m]{d_j}\sqrt[m]{d_e}=\sqrt[m]{d_e}\sqrt[m]{d_j},$
which implies $AB=BA.$ By Lemma \ref{su}, there exist matrices $A_1,\ldots,A_s$ such that $A=
\begin{pmatrix}
A_1 &        &        \\
    & A_2    &        \\
    &        & \ddots \\
    &        &        & A_s
\end{pmatrix}.$ Note that the matrices determined by $\sqrt[m]{d_k}$,
$1 \le k \neq j \le r$, all have the same shape as $A$.
Hence $s=1$; otherwise, $M$ would be a direct sum of non-trivial
$\mathbb{Z}_{+}$-submodules, which contradicts the irreducibility of $M$.
So as a $\mathbb{Z}_{+}$-module of
$\mathbb{Z}[\sqrt[m]{d_j}]$, the module $M$ is isomorphic to the direct sum of modules
$\mathbb{Z}\langle \sqrt[m]{b_{j2}b_{j3}^2\cdots b_{jm}^{m-1}}, \sqrt[m]{b_{j3}b_{j4}^2\cdots b^{m-2}_{jm}b^{m-1}_{j1}}, \ldots ,\sqrt[m]{b_{j1}b_{j2}^2\cdots b_{j(m-1)}^{m-1}} \rangle$,
where $b_{j1}\cdots b_{jm}=d_j$.

Now we prove $(1)$ by induction on $r$.
If $r=1$, then this holds by Theorem \ref{suman}. Assume that the result holds when $r=k$. If $r=k+1$, assume that $\{u_1,\ldots,u_n\}$ is a $\mathbb{Z}_{+}$-basis of $M$.
Then, as a $\mathbb{Z}_{+}$-module of
$R_1=\mathbb{Z}[\sqrt[m]{d_1},\ldots,\sqrt[m]{d_k}]$,
$M$ is a direct sum of irreducible $R_1$-submodules.
Hence, by induction, $n=m^k q$ for some positive integer $q$. Clearly from the above conclusion, the
matrix determined by $\sqrt[m]{d_1}$ has the form
$B=
\begin{pmatrix}
B_1 &        &        &        \\
    & B_1    &        &        \\
    &        & \ddots &        \\
    &        &        & B_1
\end{pmatrix}, \text{ where }
B_1=
\begin{pmatrix}
	0   & b_{1} &0 &0&\cdots&0\\
	0&0&b_{2} &0&\cdots &0\\
    \vdots & \vdots &\vdots&\vdots  & \vdots & \vdots\\
    0&0&0&0&\cdots &b_{m-1}\\
	b_{m}&0&0&0 &\cdots & 0
\end{pmatrix}$ and $b_1\ldots b_m$ are positive integers such that $b_1b_2\cdots b_m=d_1$ and not all equal. 
	Here $n=mt$. Let $\sqrt{h}\in R$ be such that $h$ is coprime to $d_1$. Suppose the matrix determined by $\sqrt[m]{h}$ is $A= \begin{pmatrix}
	A_{11} & \cdots & A_{1t}\\
	A_{21} & \cdots & A_{2t}\\
	\vdots & \cdots  & \vdots \\
    A_{t1} & \cdots & A_{tt}\\
	\end{pmatrix}$ where $A_{ij}$ is a matrix of order $m$ and $n=mt$. Since $R$ is commutative, we have $AB=BA$. By Lemma~\ref{lemma su}, we see that each $A_{ij}$ is a scalar matrix. Using the fact $A^m=dI_n,$ we observe that $A_{ij}=0$ if $i=j$ and for each $1\le i\le t$ there is exactly one $1\le j\le t$  such that $A_{ij}\ne 0.$

If $q\in \{1,2,\ldots,m-1\}$, then $n=m^kq$. Note that the number of positive integers $h$ is $m^k$ which are $m$-the power free and $\gcd(h,d_1)=1$. But $\mathrm{Rank}(R)=m^{k+1}$, so there exist at least two distinct  matrices $X$ and $Y$ which are determined by $\sqrt[m]{h_1}$ and $\sqrt[m]{h_2}$ respectively in which the non-zero entries exist at the same positions. Suppose the non-zero entries in $X$ and $Y$ are, respectively, $u_1,u_2,\ldots u_m$ and $v_1,v_2,\ldots,v_m$ \begin{comment}Let
\[ X =
\begin{pmatrix}
0   & *   & \cdots & v_1 E_3 & \cdots  & * \\
*   & 0   & \cdots & *      &\vdots & * \\
\vdots & \vdots & \ddots & \vdots  & \vdots \\
*   & 0   & \cdots & 0 & \cdots & v_2E_3 \\
\vdots & \vdots & \ddots & \vdots &\cdots & \vdots  \\
v_3 E_3 & * & \cdots & * &*& * \\
*   & *   & \cdots & * & *&0
\end{pmatrix},
 Y =
\begin{pmatrix}
0   & *   & \cdots & w_1 E_3 & \cdots  & * \\
*   & 0   & \cdots & *      &\vdots & * \\
\vdots & \vdots & \ddots & \vdots  & \vdots \\
*   & 0   & \cdots & 0 & \cdots & w_2E_3 \\
\vdots & \vdots & \ddots & \vdots &\cdots & \vdots  \\
w_3 E_3 & * & \cdots & * &*& * \\
*   & *   & \cdots & * & *&0
\end{pmatrix}\]where $h_1\ne h_2, u_1u_2u_3=h_1, v_1v_2v_3=h_2$.\end{comment} 
But then $u_1u_2\cdots u_m=h_1$ and $v_1v_2\cdots v_m=h_2$. Therefore the equality $XY=YX$ provides that $$\frac{u_1}{v_1}=\frac{u_2}{v_2}=\cdots=\frac{u_m}{v_m}.$$ 
So we have $\frac{h_1}{h_2}=\left(\frac{u_1}{v_1}\right)^m;$ which is a contradiction. Therefore, the assumption $q\in\{1,2,\ldots m-1\}$ is not possible.

Note that by Theorem~\ref{susu}, for any product
$d_{j_1}\cdots d_{j_s}$ with $1\le s\le k+1$, the matrix determined by
$\sqrt[m]{d_{j_1}\cdots d_{j_s}}$ has exactly one non-zero element in each
row and each column.
Consequently, the $\mathbb{Z}_{+}$-submodule $(m_1)$ generated by any basis
element $m_1$ has rank at most $m^{k+1}$. So $q\leq m.$
\end{proof}

\begin{corollary}
    Let $R = \mathbb{Z}[\sqrt[m]{p_{1}}, \sqrt[m]{p_{2}}, \ldots, \sqrt[m]{p_{r}}],$ where $p_i$'$s$ are distinct primes. 
Let $M$ be an irreducible $\mathbb{Z}_{+}$-module of $R$. Then $M\cong R$.
\end{corollary}
%\begin{example}Let $R = \mathbb{Z}[\sqrt[3]{4}, \sqrt[3]{3}, \sqrt[3]{35}],$ and let M be an irreducible $\mathbb Z_+$-module of $R$. Then by Theorems \ref{su} and \ref{SUSUU},   $\operatorname{rank}(M)=27$ and up to isomorphism, there are exactly $4$ irreducible $\mathbb{Z}_{+}$-module of $R$.\end{example}


\begin{thebibliography}{99}


\bibitem{CHI1}
Z.~Chen and R.~Zhao,
Classification of irreducible $\mathbb{Z}_{+}$-modules of a $\mathbb{Z}_{+}$-ring using matrix equations,
Symmetry \textbf{14} (2022), 2598.

\bibitem{CHI2}
Y.~Chen, L.~Li, and Y.~Zheng,
Irreducible $\mathbb{Z}_{+}$-modules over a class of domains,
Commun. Algebra \textbf{53} (2025), no.~1, 11--17.

\bibitem{NO}
P.~Etingof, D.~Nikshych, and V.~Ostrik,
On fusion categories,
Ann. of Math. (2) \textbf{162} (2005), no.~2, 581--642.

\bibitem{ETIN}
P.~Etingof and M.~Khovanov,
Representations of tensor categories and Dynkin diagrams,
Internet Math. Res. Not. \textbf{5} (1995), 235--247.

\bibitem{Lusztig1987}
G.~Lusztig,
Leading coefficients of character values of Hecke algebras,
Proc. Sympos. Pure Math. \textbf{47} (1987), 235--262.

\bibitem{OST1}
V.~Ostrik,
Module categories, weak Hopf algebras and modular invariants,
Transform. Groups \textbf{8} (2003), 177--206.

\bibitem{CHI}
C.~Yuan, R.~Zhao, and L.~Li,
Irreducible $\mathbb{Z}_{+}$-modules of near-group fusion ring $K(\mathbb{Z}_3,3)$,
Front. Math. China \textbf{13} (2018), 947--966.

\end{thebibliography}
\end{document}